\documentclass[12pt]{article}
\usepackage[utf8]{inputenc}
\usepackage{amsmath,amssymb,amsthm}
\renewcommand\le{\leqslant}
\renewcommand\ge{\geqslant}

\newcommand\T{\mathbb{T}}
\newcommand\R{\mathbb{R}}
\newcommand\E{\mathsf{E}}
\newcommand\eps{\varepsilon}
\renewcommand\P{\mathsf{P}}

\newcommand\LH{L_{\mathrm{H}}}
\newcommand\avg{\mathrm{avg}}

\DeclareMathOperator\Rig{Rig}
\DeclareMathOperator\rank{rank}
\DeclareMathOperator\vc{vc}
\DeclareMathOperator\tr{tr}
\DeclareMathOperator\Vol{Vol}
\DeclareMathOperator\Law{Law}
\DeclareMathOperator\supp{supp}
\DeclareMathOperator\sign{sign}

\newtheorem{exttheorem}{Theorem}

\newtheorem*{theorem*}{Theorem}
\newtheorem{statement}{Statement}[section]
\newtheorem{theorem}{Theorem}[section]
\newtheorem{lemma}{Lemma}[section]
\newtheorem{corollary}{Corollary}[section]
\newtheorem*{example}{Example}
\newtheorem*{remark}{Remark}
\newtheorem{question}{Question}

\title{Widths and rigidity}
\author{Yuri Malykhin}

\begin{document}
\maketitle

\begin{abstract}
    We consider Kolmogorov widths of finite sets of functions.
    Any orthonormal system of $N$ functions is \textit{rigid} in $L_2$, i.e.
    it cannot be well approximated by linear subspaces of
    dimension essentially smaller than $N$.
    This is not true for weaker metrics: it is known that in every $L_p$, $p<2$,
    the first $N$ Walsh functions can be $o(1)$-approximated by a linear space of
    dimension $o(N)$.

    We give some sufficient conditions for rigidity.
    We prove that independence of functions (in the probabilistic meaning) implies rigidity
    in $L_1$ and even in $L_0$~--- the metric that corresponds to convergence in
    measure. In the case of $L_p$, $1<p<2$, the condition is weaker:
    any $S_{p'}$-system is $L_p$-rigid.

    Also we obtain some positive results, e.g. that first $N$ trigonometric
    functions can be approximated by very-low-dimensional spaces in $L_0$,
    and by subspaces generated by $o(N)$ harmonics in $L_p$, $p<1$.
\end{abstract}

\section{Introduction}

In this paper we consider Kolmogorov widths of finite sets of functions.
Recall that the Kolmogorov width
$d_n(K,X)$ is defined as the minimum distance $\sup\limits_{x\in K}\rho(x,Q_n)_X$ from
the set $K$ to $n$-dimensional subspaces $Q_n$ of the space $X$. This classical
notion goes back to the paper~\cite{Kol36} by A.N.~Kolmogorov (1936). Widths are
fundamental approximation characteristic of sets 
and the behaviour of the sequence $(d_n(K,X))_{n=1}^\infty$ for functional
classes and other sets of interest was actively studied by many authors.
The books~\cite[Ch.13, 14]{LGM96} and~\cite{P85} are standard references
on the subject. See also a recent survey~\cite{DTU18}, \S4.3.

We will informally say that a set is \textit{rigid}, if it cannot be well
approximated by low-dimensional subspaces.
The starting point for us is the well
known result that any orthonormal system of functions
$\varphi_1,\ldots,\varphi_N$ is rigid in $L_2(0,1)$
(of course, this holds true for any Euclidean space):
\begin{equation}
    \label{l2width}
d_n(\{\varphi_1,\ldots,\varphi_N\},L_2) = \sqrt{1-n/N}.
\end{equation}
(This equality for widths was first used by S.B.~Stechkin, see~\cite{Tikh87} for
details.)

The situation changes if we consider weaker metrics.
\begin{exttheorem}[\cite{Mal22}]
    \label{thm_walsh_p12}
    Let $w_1,w_2,\ldots$ be the Walsh system in the Paley numeration. For any
    $p\in[1,2)$ there exists $\delta=\delta(p)>0$ such that for sufficiently
    large $N$ the inequality holds
    $$
    d_n(\{w_1,\ldots,w_N\},L_p[0,1])\le N^{-\delta},\quad
    \mbox{if $n\ge N^{1-\delta}$}.
    $$
\end{exttheorem}
This theorem appeared as a corollary of methods introduced in Complexity theory
for the related problem of low-rank approximation, namely, \textit{matrix
rigidity}. The rigidity function of a matrix $A$ is defined as the Hamming distance from $A$ to
matrices of given rank:
$$
\Rig(A,n) := \min_{\rank B\le n} \#\{(i,j)\colon A_{i,j}\ne B_{i,j}\}.
$$
Matrix rigidity was introduced by L.Valiant in 1977 as a way to
obtain lower bounds for the linear circuit complexity (e.g. the inequality
$\Rig(A_N,\eps N)\ge N^{1+\eps}$ gives superlinear lower bounds for the circuit
complexity), see the survey~\cite{L08} for more details.

It was a surprising result by Alman and Williams~\cite{AW17} that the family of
Walsh--Hadamard matrices is not rigid, i.e., the rank of that matrices may be
significantly decreased by changing a small amount of elements.
Theorem~\ref{thm_walsh_p12} is based on their construction.

In this paper we get sufficient conditions for systems to be rigid in 
metrics weaker than $L_2$.

\paragraph{Average widths and $L_1$-rigidity.}
Our lower bounds on widths are based on the approximation properties of random vectors in $\R^N$. 
    Let $\xi_1,\ldots,\xi_N$ be independent random variables with $\E\xi_i=0$
    and $\E|\xi_i| = 1$. We prove that for any $n$-dimensional space
    $Q_n\subset\R^N$,
    $$
    \E\rho(\xi,Q_n)_{\ell_1^N} \ge c(\eps)N, \quad\mbox{if $n\le N(1-\eps)$}.
    $$
Let us introduce the notation
$$
d_n^\avg(\xi,X) := \inf_{\dim Q_n\le n}\E\rho(\xi,Q_n)_{X}
$$
and call this quantity the ``average Kolmogorov width'' of a random
vector $\xi$. If fact, average widths were first considered by 
S.M.~Voronin and N.T.~Temirgaliev in~\cite{VT84}. This notion was further
studied by V.~Maiorov~\cite{M94}, Creutzig~\cite{Cr02} and others. We discuss
the notation $d_n^\avg$ and the properties of the average widths in
Section~\ref{subsec_avg_widths}.  There is a simple equality
$$
d_n^\avg(\xi,\ell_1^N) = Nd_n^\avg(\{\xi_1,\ldots,\xi_N\},L_1(\Omega))
$$
that links the rigidity of finite-dimensional vectors and the rigidity of finite
function systems. Note that the averaging on the right side is over finite set
of elements of $L_1$, i.e. we minimize
$\frac1N\sum_{k=1}^N\rho(\xi_k,Q_n)$ over subspaces of $L_1$.

So, we can say that \textit{independence implies rigidity}:
\begin{theorem}
    \label{th_l1_intro}
    Let $\xi_1,\ldots,\xi_N$ be independent random variables with $\E\xi_i=0$
    and $\E|\xi_i| = 1$. Then
    $$
    d_n^\avg(\{\xi_1,\ldots,\xi_N\},L_1(\Omega)) = N^{-1}d_n^\avg(\xi,\ell_1^N)
    \ge c(\eps)>0, \quad\mbox{if $n\le N(1-\eps)$}.
    $$
\end{theorem}
This also gives the rigidity for the standard widths, because $d_n \ge d_n^\avg$.
A classical example of an independent system is Rademacher system.

The same result holds if we replace the independence by the
unconditionality:
$\Law(\xi_1,\ldots,\xi_N)=\Law(\pm\xi_1,\ldots,\pm\xi_N)$.

\paragraph{Approximation in the space of measurable functions.}
Consider the space $L_0(\mathcal X,\mu)$ of all measurable functions on some measure
space $(\mathcal X,\mu)$. The convergence in measure is equivalent to
convergence in certain metrics, say
$\int_{\mathcal X}\frac{|f-g|}{1+|f-g|}\,d\mu$ or $\sup\{\eps\colon
\mu(|f-g|\ge\eps)\ge\eps\}$. We will use the latter. Of course, this metric is
weaker than usual $L_p$-metrics. The approximation in $L_0$ is much less studied
than in $L_p$ but we believe that this subject contains interesting
problems.

In Section~\ref{subsec_l0} we will prove that independence implies rigidity
even in the case of $L_0$ metric.
\begin{theorem}
    For any $\eps\in(0,1)$ there exists $\delta>0$, such that if $\xi_1,\ldots,\xi_N$
    are independent r.v. with $\inf_{c\in\R}\|\xi_i-c\|_{L_0}\ge\eps$,
    $i=1,\ldots,N$. Then
    $$
    d_n^{\mathrm{avg}}(\{\xi_1,\ldots,\xi_N\},L_0)\ge\delta\quad\mbox{if $n\le
    \delta N$.}
    $$
\end{theorem}
This is also a corollary of a finite-dimensional result. In fact there is an exponential bound
$$
\P(\rho(\xi,Q_n)_{L_0^N} \le \delta)\le 2\exp(-\delta N),\quad\dim Q_n\le n.
$$
Let us mention a recent result by B.S.~Kashin~\cite{K21}.
\begin{exttheorem}
    There exist absolute constants $0<\gamma_0<1$, $c_1,c_2>0$, such that for
    $N=1,2,\ldots$ for any orthonormal basis $\Psi=\{\psi_k\}_{k=1}^N$ in $\R^N$
    and any linear subspace $Q\subset\R^N$, $\dim Q\le c_1 N$, the inequality
    holds
    $$
    \P\{\eps\in\{-1,1\}^N\colon \rho(\sum_{k=1}^N\eps_k\psi_k,Q)_{L_0^N}\le c_2\}\le
    \gamma_0^N.
    $$
\end{exttheorem}

It is known that lacunary series $\sum c_j\varphi(k_jx)$ behave like sums of
independent variables, see e.g.~\cite{Gap66}. The rigidity property is not an
exception. We will prove that lacunary systems
$\{\varphi(k_1x),\ldots,\varphi(k_Nx)\}$ are rigid in $L_0$ if $\min
k_{j+1}/k_j$ is large enough, see Statement~\ref{stm_lacunary}.

Another source of rigid systems is random matrices. It is well-known that a
square $N\times N$ matrix $\mathcal E$ consisting of random signs is rigid with high probability:
$\Rig(\mathcal E,\delta N)\ge \delta N^2$. We observe that this is true in weaker
$L_0$-metric:
$$
\P(\min_{\rank B\le\delta N}\|\mathbf{\mathcal E}-B\|_{L_0^{N\times
N}}<\delta)\le 2\exp(-\delta N^2).
$$
This gives us an example of a $L_0$-rigid orthonormal system
$\{\varphi_1,\ldots,\varphi_N\}$ of functions that are piecewise constant on
intervals $(\frac{j-1}N,\frac{j}N)$.

There is another metric in $L_0(\mathcal X,\mu)$, namely, $\mu\{x\colon f(x)\ne g(x)\}$.
We denote it by $\LH$; the distance $\|f-g\|_{\LH}$ is the analog of Hamming
distance. (In Probability it is
called the ``indicator distance''.)
This metric appears in Function theory in the context of ``correction''
theorems. E.g., Lusin's theorem states that $C$ is dense in $\LH$ and Menshov's
theorem implies that the set of functions with uniformly convergent
Fourier series is dense in $\LH$.
Note that the matrix rigidity is a particular case of average
Kolmogorov widths in normalized Hamming distance $\LH^N$. Let $A$ be an $M\times N$ real
matrix and $W_A=\{A_1,\ldots,A_M\}\subset\R^N$ be the set of rows of $A$. Then
$$
\Rig(A,n) = M\cdot d_n^{\mathrm{avg}}(W_A,\LH^N).
$$

Some of our results may be stated in matrix terms; in some cases matrix language
is even more convenient. However, in this paper we prefer to stick to the
functional-theoretic language. See also~\cite{KMR18}.

\paragraph{$S_{p'}$-property.}
In the case $1<p<2$ we establish a less restrictive sufficient condition for
rigidity.
\begin{theorem}
    Let $\varphi_1,\ldots,\varphi_N$ be an orthonormal system in a space
    $L_2(\mathcal X,\mu)$, $\mu(\mathcal X)=1$.  Let $1<p<2$ and suppose that $\{\varphi_k\}$ has
    $S_{p'}$-property, i.e. for some $B\ge1$ one has
    $$
    \|\sum_{k=1}^N a_k\varphi_k \|_{p'} \le B|a|,\quad\forall a\in\R^N.
    $$
    Then this system is rigid in $L_p$:
    $$
    d_n^{\mathrm{avg}}(\{\varphi_1,\ldots,\varphi_N\},L_p)_p \ge
    c(B,\eps,p)>0,\quad\mbox{if $n\le N(1-\eps)$}.
    $$
\end{theorem}

Note that Bourgain's theorem~\cite{B89} states that any uniformely bounded orthornomal
system has a $S_{p'}$-subsystem of size $N^{2/p'}$; this subsystem is rigid in
$L_p$. Hence the $o(1)$-approximation of the uniformely bounded orthornomal
system in $L_p$ requires the dimension at least $n\ge \frac12 N^{2/p'}$. In
particular, in Theorem~\ref{thm_walsh_p12} we indeed should have the polynomial
dependence between $n$ and $N$.

The analog of Theorem~\ref{th_l1_intro} also holds true for $1<p\le 2$ (we do not
prove this here) but it fails for $p>2$.

\paragraph{Positive approximation results.}

Let $G$ be a finite Abelian group and $\widehat{G}$ be the dual group of
characters $\chi\colon G\to\mathbb C^*$.

From~\cite{DL19} it follows that characters are not rigid in $\LH$:
$$
d_n(\widehat{G},\LH)\le N^{-1+c\eps},\quad
\mbox{if $n\ge N\exp(-c(\eps)\log^{c_2}N)$, $N:=|G|$}.
$$
Such small error of approximation, $N^{-1+c\eps}$, is natural in the context of Matrix
Rigidity due to the Valiant's result on circuit complexity.
However, we are interested in approximation in different ``regimes''. We
consider systems generated by specific characters, namely, Walsh and trigonometric systems. First we
show that they admit good approximation by very-low-dimensional spaces.
\begin{theorem}
    We have
    $$
    d_n(\{w_1,\ldots,w_N\},\LH[0,1])\le \exp(-c\log^{1/4}N),\quad
    \mbox{if $n\ge\exp(C\log^{2/3}N)$.}
    $$
    $$
    d_n(\{e(kx)\}_{k=-N}^N, L_0(\mathbb T))\le \exp(-c\log^{0.2}N),\quad
    \mbox{if $n\ge\exp(C\log^{0.9}N)$.}
    $$
\end{theorem}

It is known that any harmonic can be approximated 
by other harmonics in $L_p$, $p<1$.
\begin{exttheorem}[\cite{IvYu80}]
    For any $0<p<1$ we have uniformly in $M$ and $N\ge M$,
    $$
    \inf\|\cos Mx-\sum_{|k|\le N,\,k\ne M}c_ke(kx)\|_p \asymp (N-M+1)^{1-1/p}.
    $$
\end{exttheorem}

We bound the \textit{trigonometric widths} $d_n^T$ of the first $N$
trigonometric functions, i.e. approximate them
by trigonometric subspaces $Q_n=\{\sum_{k\in \Lambda}c_ke(kx)\}$, $|\Lambda|=n$.
\begin{theorem}
    For any $p\in(0,1)$ and for sufficiently large $N$ we have
    $$
    d_n^T(\{e(kx)\}_{k=-N}^N,L_p) \le \log^{-c_1(p)}N,\quad\mbox{if $n\ge
    N\exp(-\log^{c_2}N)$},
    $$
    $$
    d_n^T(\{e(kx/N)\}_{k\in \mathbb Z_N}, L_p^N) \le \log^{-c_1(p)}N,\quad\mbox{if $n\ge
    N\exp(-\log^{c_2}N)$}.
    $$
\end{theorem}

\paragraph{Organization of the paper.}
In the next section we provide necessary definitions and some needed facts. In
Sections~\ref{sec_l01},~\ref{sec_lp} and~\ref{sec_positive} we prove the results
mentioned in the Introduction as well as some other results. In the last
Section we discuss remaining questions on the subject.

\paragraph{Acknowledgement.}
Author thanks Konstantin Sergeevich Ryutin, Sergei Vladimirobich Astashkin and
Boris Sergeevich Kashin for fruitful discussions.

\section{Preliminaries}

\subsection{Metrics}
\label{subsect_metrics}

Let $(\mathcal X,\mu)$ be a measure space, $\mu(\mathcal X)<\infty$.
As usual, we identify functions (or random variables) that differ on a set of
zero measure.

We will consider classical
$L_p$-spaces for $0<p<\infty$:
$$
\|f\|_p := \|f\|_{L_p} := \left(\int_{\mathcal X}|f|^p\,d\mu\right)^{1/p}.
$$

The space $L_0(\mathcal X,\mu)$ consists of all measurable functions. The convergence in measure
in $L_0$ may be metricised by several metrics. Let us describe the Ky-Fan metric
that we will use. Denote
$$
\|f\|_{L_0} := \sup\{\eps\ge0\colon \mu\{x\colon |f(x)|\ge\eps\}\ge\eps\}.
$$
Then $\|f-g\|_{L_0}$ is indeed a metric.
Also we use the following notation:
$$
\|f\|_{\LH} := \mu \{x\colon f(x)\ne0\}.
$$
Note that $\|f\|_p \ge \|f\|_{L_0}^{1+1/p}$ and $\|f\|_{\LH}\ge\|f\|_{L_0}$.

Let $N\in\mathbb N$. Then the space $\R^N$ is equipped with the usual norms
$\ell_p^N$ that correspond to the counting measure on $\{1,\ldots,N\}$:
$$
\|x\|_p := \|x\|_{\ell_p^N} := (\sum_{k=1}^N |x_k|^p)^{1/p}.
$$
We denote the unit ball in $\ell_p^N$ as $B_p^N$.
Note that the meaning of $\|\cdot\|_p$ is clear from the context: it is either
the $\|\cdot\|_{L_p}$ norm if the argument is a function and $\|\cdot\|_{\ell_p^N}$
norm if the argument is a vector. In the Euclid case we write $|\cdot|$.

By $L_p^N$ we denote the $L_p$-spaces for normalized (probabilistic) measure on
$\{1,\ldots,N\}$. Obviously, $\|x\|_{L_p^N} = N^{-1/p}\|x\|_{\ell_p^N}$,
$0<p\le\infty$.

The quantity $\|x\|_{\LH^N}$ is the share of nonzero coordinates
of a vector and the corresponding distance is the normalized
Hamming distance.
Note that $\lim\limits_{p\to+0}\|x\|_p^p = \#\{k\colon x_k\ne0\} = N\|x\|_{\LH^N}$.

In the case $p=0$ we do not work with $\ell_N^0$, but only with $L_0^N$:
$$
\|x\|_{L_0^N} = \max\{\eps\colon \#\{i\colon |x_i|\ge \eps\} \ge \eps N\}.
$$
We avoid the notation $\|\cdot\|_0$.

\subsection{Widths}
\label{subsec_avg_widths}

Let $X$ be a linear space equipped with a metric (or more general distance
functional) $d_X$. Recall that the Kolmogorov width of a set $K\subset X$ of order
$n\in\mathbb Z_+$ is defined as
$$
d_n(K,X) := \inf_{Q_n\subset X} \sup_{x\in K}\rho(x,Q_n)_X,
$$
where $\rho(x,Q)_X := \inf_{y\in Q} d_X(x,y)$ and the infimum is taken over
linear subspaces of $X$ of dimension at most $n$. Usually $X$ is a normed space
and $d_X(x,y)=\|x-y\|_X$; note that in the original Kolmogorov's paper this is
not required and we will indeed use widths in more general context.

\paragraph{Average widths.}
Let $X$ be a linear space with a metric, $\mu$~--- a Borel measure on $X$. For $p\in(0,\infty)$ we define the
$p$-averaged Kolmogorov width of $\mu$ in $X$ as
$$
d_n^{\mathrm{avg}}(\mu,X)_p := \inf_{Q_n\subset X} (\int_X \rho(x,Q_n)_X^p \,d\mu(x))^{1/p}.
$$
If $p=1$, we write simply $d_n^{\mathrm{avg}}(\mu,X)$.

From now on we will assume that $\mu$ is a probability measure. Equivalently, we
may consider the widths of a (Borel) random vector $\xi$ with values in $X$:
$$
d^{\mathrm{avg}}_n(\xi, X)_p := \inf_{Q_n\subset X} (\E \rho(\xi,Q_n)_X^p)^{1/p}.
$$
If $K$ is a subset in $X$ and it is clear, what is meant by a ``random point in
$K$'', then the width $d_n^{\mathrm{avg}}(K,X)_p$ is defined as a width of that
random point. E.g. for a finite $K$ we have
$$
d^{\mathrm{avg}}_n(\{x_1,\ldots,x_N\}, X) := \inf_{Q_n\subset X} \frac1N
\sum_{i=1}^N \rho(x_i,Q_n)_X.
$$

\paragraph{Notation.}
Let us discuss the new notation $d_n^{\mathrm{avg}}(\xi,X)_p$.
Some authors write $d_n^{(a)}$ instead of $d_n^{\mathrm{avg}}$. This may be confusing
because $d_n^{a}$ stands for \textit{absolute} widths (see~\cite{OO13}).
Also the usage of $\mathrm{avg}$ is standard
in average settings in Information-based complexity, see~\cite{NWBook}.
Some authors write $d_n^{(a)}(X,\mu)$, but
when we deal with sets and random variables the notation
$d_n^{\mathrm{avg}}(\xi,X)$ seems to be more intuitive (the object which width
is considered is written in the first place).


\paragraph{Properties.}
The most basic inequality that we always have in mind is
$$
d_n(K,X) \ge d_n^\avg(\xi,X)_p,\quad\mbox{if $\P(\xi\in K)=1$.}
$$
We will use the following inequality, which is a simple consequence of the convexity
of the $L_p$-norm:
\begin{equation}
    \label{sum_width}
d_{n_1+n_2}^{\mathrm{avg}}(\xi_1+\xi_2,X)_p \le
d_{n_1}^{\mathrm{avg}}(\xi_1,X)_p +
d_{n_2}^{\mathrm{avg}}(\xi_2,X)_p,\quad 1\le p<\infty.
\end{equation}

\begin{statement}
    \label{stm_equiv}
    Let $\xi_1,\ldots,\xi_N\colon\Omega\to\R$ be some random variables with
    $\E|\xi_i|^p<\infty$ for all $i$.
    Then for $p\in[1,\infty)$ and $n\in\mathbb Z_+$ we have
    \begin{equation}
        \label{avg_equiv}
    d_n^{\mathrm{avg}}(\{\xi_1,\ldots,\xi_N\},L_p(\Omega))_p =
    N^{-1/p}d_n^{\mathrm{avg}}((\xi_1,\ldots,\xi_N),\ell_p^N)_p.
    \end{equation}
\end{statement}

Note that in the width on the left side the averaging is over finite set of
``points'' $\xi_1,\ldots,\xi_N$ in the space $L_p(\Omega)$; and on the right side
we see the averaged width of a random vector $\xi=(\xi_1,\ldots,\xi_N)$ in $\R^N$.
In the case of discrete $\Omega$ the equality~\eqref{avg_equiv} is a simple consequence of the equality of
the row-rank and the column-rank of a matrix. Let us give the full proof for
completeness.

\begin{proof}
    Consider the error of an approximation $\xi_k\approx \eta_k$:
    \begin{equation}
        \label{xi_eta_transpose}
    \sum_{k=1}^N \|\xi_k-\eta_k\|_{L_p(\Omega)}^p = \sum_{k=1}^N
    \E|\xi_k-\eta_k|^p = \E \|\xi-\eta\|_{\ell_p^N}^p.
    \end{equation}
    Note that random variables $\eta_1,\ldots,\eta_N$ lie in a subspace of
    $L_p(\Omega)$ of dimension at most $n$ if and only if
    there is a subspace $Q\subset\R^N$, $\dim Q\le n$, such that
    $$
    \eta = (\eta_1,\ldots,\eta_N)\in Q\quad \mbox{almost surely}.
    $$
    Therefore,
    $$
    \E \|\xi-\eta\|_{\ell_p^N}^p \ge \E\rho(\xi,Q)_{\ell_p^N}^p \ge
    d_n^{\mathrm{avg}}(\xi,\ell_p^N)_p^p.
    $$
    This gives us the ``$\ge$'' inequality in~\eqref{avg_equiv}.

    To prove the reverse inequality, we have to construct $\eta$ from an
    (almost) optimal subspace $Q\subset\R^N$. Note that for any $\eps>0$ there
    is a Borel-measurble (e.g., piecewise-constant) map $\pi\colon\R^N\to Q$, such that
    $$
    \|x-\pi(x)\|_{\ell_p^N} \le \rho(x,Q)_{\ell_p^N} + \eps,\quad x\in\R^N.
    $$
    Take $\eta := \pi(\xi)$; this is a random vector because $\pi$ is
    Borel. Moreover, $\eta\in Q$ almost surely, so $\eta_1,\ldots,\eta_N$ span at most
    $n$-dimensional space in $L_p(\Omega)$.
    It remains to substitute $x=\xi$ in the previous
    inequality, take $p$-expectation and use~\eqref{xi_eta_transpose} to finish the proof.
\end{proof}

\paragraph{Average width in Hilbert space.}

We will give a formula for 2-averaged width of a random vector in Hilber space.
This formula is a generalization of Eckart--Young theorem on matrix
approximation in Frobenius norm. Also, it is a reformulation of Ismagilov's
theorem~\cite{I68}. However, our terminology is very convenient so this
formulation makes sense. Our proof repeats Ismagilov's and is given for
completeness.

Let $H$ be a separable Hilber space (of finite or infinite dimension), $\xi$ a
random vector in $H$ such that $\E|\xi|^2<\infty$. Consider the correlation
operator of $\xi$ defined by $\langle Au,v\rangle =
\E\langle\xi,u\rangle\langle\xi,v\rangle$. It is known that $A$ is symmetric
nonnegative compact operator (see, e.g.~\cite[Ch.3]{VCT}). The sequence
$(\lambda_n)$ is the sequence of eigenvalues of $A$ if and only if there is an
orthonormal basis $\{\varphi_n\}$ (eigenbasis of $A$) such that
\begin{equation}
    \label{bazis_H}
\E\langle\xi,\varphi_k\rangle^2=\lambda_k,\quad
\E\langle\xi,\varphi_k\rangle\langle\xi,\varphi_l\rangle=0\;\mbox{for $k\ne l$}
\end{equation}
Let us consider the eigenvalues in the decreasing order:
$\lambda_1\ge\lambda_2\ge\ldots\ge0$.

\begin{statement}
    We have $d_n^\avg(\xi,H)_2 = (\sum\limits_{k>n}\lambda_k)^{1/2}$.
\end{statement}

\begin{proof}
    Fix a basis~\eqref{bazis_H} and put $\xi_k:=\langle\xi,\varphi_k\rangle$.
    Let us bound the width from below; consider some approximating space $Q_n$
    with orthonormal basis $\{\psi_1,\ldots,\psi_n\}$. Then
    $$
    \E\rho(\xi,Q_n)^2 = \E|\xi|^2 - \E|P_{Q_n}\xi|^2.
    $$
    We have $\E|\xi|^2=\sum_{k\ge1}\E\xi_k^2=\sum_{k\ge
    1}\lambda_k$ so it is enough to give an upper bound on the projection:
    \begin{multline*}
    \E|P_{Q_n}\xi|^2 = \sum_{j=1}^n \E\langle \xi,\psi_j\rangle^2 =
    \sum_{j=1}^n\E (\sum_{k\ge 1}\xi_k\langle \psi_j,\varphi_k\rangle)^2 = \\
    = \sum_{j=1}^n \sum_{k\ge1} \lambda_k \langle \psi_j,\varphi_k\rangle^2 = \sum_{k\ge
        1}\lambda_k \mu_k,\quad\mbox{where $\mu_k := \sum\limits_{j=1}^n
        \langle\psi_j,\varphi_k\rangle^2$}.
    \end{multline*}
    Note that numbers $(\mu_k)$ belong to $[0,1]$ and their sum equals $n$.
    Due to monotonicity of $\lambda_n$ the sum $\sum_{k\ge1}\lambda_k\mu_k$ is
    maximal for $\mu_k = 1$, $k\le n$, and $\mu_k=0$, $k>n$. We obtained the
    lower bound for the width. The upper bound is attained at the subspace
    $Q_n = \mathrm{span}\,\{\varphi_1,\ldots,\varphi_n\}$.
\end{proof}

It is often convenient to work with singular values $\sigma_k:=\lambda_k^{1/2}$
instead of eigenvalues. One can rewrite~\eqref{bazis_H} in terms of Schmidt
decomposition:
$$
\xi = \sum_{k\ge 1}\sigma_k \eta_k \varphi_k,
\quad \E\eta_k\eta_l=\delta_{k,l},
\quad \langle\varphi_k,\varphi_l\rangle=\delta_{k,l}.
$$
Then $d_n^\avg(\xi,H)_2 = (\sum_{k>n}\sigma_k^2)^{1/2}$.

We come to Eckart--Young theorem if we consider finite sets. Let matrix $X$ consist of
columns $x^1,\ldots,x^N\in\R^M$, then
$$
d_n^\avg(\{x^1,\ldots,x^N\},\ell_2^M)_2 =
N^{-1/2}(\sum_{k>n}\sigma_k(X)^2)^{1/2}.
$$
Indeed, singular value decomposition of $X$ gives a basis
$\{\varphi_k\}_{k=1}^M$, such that coordinate vectors $(\langle
x^1,\varphi_k\rangle,\ldots,\langle x^N,\varphi_k\rangle)$ are orthogonal and
have length $\sigma_k(X)$. Given $\xi$~--- a random choice of  $x^i$, we have
$\E\langle\xi,\varphi_k\rangle^2=N^{-1}\sigma_k^2(X)$,
$\E\langle\xi,\varphi_k\rangle\langle\xi,\varphi_l\rangle=0$, hence $\sigma_k(\xi) =
N^{-1/2}\sigma_k(X)$.

Finally, an orthonormal system
$\varphi_1,\ldots,\varphi_N$ in Euclid space
$E$ corresponds to the identity matrix $X$ (in appropriate coordinates), so we
have (see also~\eqref{l2width}):
$$
d_n(\{\varphi_1,\ldots,\varphi_N\},E) =
d_n^\avg(\{\varphi_1,\ldots,\varphi_N\},E)_2 = \sqrt{1-n/N}.
$$

\paragraph{Trigonometric width.}
Let $\Phi$ be a subset of $X$. If we approximate by subspaces spanned by
elements of $\Phi$, we obtain the notion of $\Phi$-width
$$
d_n^\Phi(K,X) := \inf_{\varphi_1,\ldots,\varphi_n\in\Phi}\sup_{x\in
K}\rho(x,\mathrm{span}\,\{\varphi_i\}_{i=1}^n)_X.
$$
However we will need only the case when $\Phi$ is a trigonometric system
$\{e(kx)\}_{k\in\mathbb Z}$ in the continuous case or
$\{e(kx/N)\}_{k\in\mathbb Z_N}$ in the discrete case. In this situation the
width is called the trigonometric width and is denoted as $d_n^T$.

Usually we consider real spaces (and the linear dimension over $\mathbb R$).
When complex-valued functions are approximated one can consider also the linear
spaces over $\mathbb C$. The distinction here will be not important for us
because in all cases where complex--valued functions are approximated the
precise value of the dimension $n$ plays no role.

\subsection{Miscellaneous}

The material from the following two paragraphs is contained in the book~\cite[\S2.2,\,2.8,\,8.8]{HDPBook}.

\paragraph{Probability.}
We use probabilistic notations $\P$ for the Probability, $\E$ for the
Expectation, $\Law$ for the distribution. We also use the conditional expectation $\E(\xi|\eta)$; however,
only for $\eta$ with finite number of values.

We will make use of standard tail estimates.
Let $X_i$ be independent r.v., $\E X_i=0$.

\textit{Hoeffding's inequality}. Let $a_i\le X_i\le b_i$ a.s., then
$$
\P(|\sum_{i=1}^k X_i| \ge \lambda \sigma) \le
2\exp(-2\lambda^2),\quad \sigma^2 := \sum_{i=1}^k (b_i-a_i)^2.
$$

\textit{Bernstein's inequality}. Let $|X_i|\le M$ a.s., then
$$
\P(|\sum_{i=1}^k X_i| \ge t) \le
2\exp\left(-\frac{\frac12 t^2}{\sigma^2+\frac13Mt}\right),\quad \sigma^2 := 
\sum_{i=1}^n \E X_i^2.
$$

\paragraph{$\vc$-dimension.}
Let us recall the definition of the \textit{combinatorial dimension}
of some class $\mathcal F$ of functions $f\colon I\to\R$. We say
that a set $J\subset I$ is $t$-shattered ($t>0$) by $\mathcal F$, if there is a function $h\colon
J\to\R$, such that for any choice of signs $\sigma\colon J\to\{-1,1\}$ one can
find a function $f_\sigma\in\mathcal F$ with
\begin{equation}
    \label{shatter}
\min_{x\in J}\sigma(x)(f_\sigma(x)-h(x))\ge t.
\end{equation}
Then the combinatorial dimension $\vc(\mathcal F,t)$ is the maximal cardinality
of a set $t$-shattered by $\mathcal F$.

If $\mathcal F$ is convex and centrally-symmetic then one can assume that
$h(x)\equiv 0$. Indeed, we can replace $\{f_\sigma\}$ by
$\{\frac12(f_\sigma-f_{-\sigma})\}$.

We will use a simple (and known) fact that
\begin{equation}
    \label{vc_width}
\mbox{if\;\;}\vc(\mathcal F,t)> n,\quad\mbox{then\;\;}d_n(\mathcal F,\ell_\infty(I))\ge t.
\end{equation}
Indeed, let $J$ be a $t$-shattered set, $|J|=n+1$, and suppose that there is an
approximation by an $n$-dimensional subspace:
$\mathcal F\ni f\approx g$, $g\in Q_n$, $\|f-g\|_\infty < t$.
The exists a nontrivial linear functional
$\Lambda(f) := \sum_{x\in J}\lambda(x)f(x)$, such that $\Lambda(g)\equiv 0$ for
$g\in Q_n$.
Assume that $\Lambda(h)\ge 0$, where $h$ is from~\eqref{shatter}. Take a
function $f$ such that
$\sign(\lambda(x))(f(x)-h(x))\ge t$ for all $\{x\in J\colon \lambda(x)\ne0\}$.
But for the approximating function we have $\sign(\lambda(x))(g(x)-h(x))>0$,
hence $\Lambda(g)=\Lambda(g-h)+\Lambda(h)>0$. We get a contradiction. The case $\Lambda(h)\le0$
is analogous.

We will also
make use of the classical VC bound for $\mathcal F\colon I\to\{-1,1\}^k$:
\begin{equation}
    \label{vc_std}
|\mathcal F|>\binom{k}{0}+\binom{k}{1}+\ldots+\binom{k}{d}
\quad\Longrightarrow\quad \vc(\mathcal F,1) > d.
\end{equation}
There is a useful bound on binomial coefficients:
\begin{equation}
    \label{binom}
\binom{k}{0}+\ldots+\binom{k}m \le (ek/m)^m,\quad 0\le m \le k.
\end{equation}
Note that the function $(ek/x)^x$ is increasing for $x\in(0,k]$.

We will ofter consider $1$-periodic functions, i.e. functions on $\mathbb
T:=\mathbb R/\mathbb Z$. We also use the notation $e(\theta):=\exp(2\pi
i\theta)$.

By $c,c_1,c_2,C,\ldots$ we denote positive constants. If there is a dependence
on some parameters, we specify them explicitly: $c_p,c(\eps),\ldots$.

\section{Sufficient conditions for rigidity in weak metrics}
\label{sec_l01}

\subsection{Rigidity in $L_1$}

For definiteness we put $\sign(x)=1$ for $x\ge 0$ and $\sign(x)=-1$ for
$x<0$.

\begin{theorem}
    \label{th_l1}
    Let $\xi=(\xi_1,\ldots,\xi_N)$ be a random vector in $\R^N$ having
    $\E|\xi_i|\ge1$ for all $i$. Let $\eps\in(0,1)$ and $n\le N(1-\eps)$. Then
    \begin{equation}
        \label{l1rig}
        d_n^{\mathrm{avg}}(\xi,\ell_1^N) \ge
        c(\eps)N - \sum_{i=1}^N\|\E(\xi_i|\{\sign\xi_j\}_{j\ne i})\|_{L_1}.
    \end{equation}
\end{theorem}

Before proving this Theorem let us note that the distance in $\ell_1^N$ from a
vector $x\in\R^N$ to a subspace $Q_n\subset\R^N$ may be written in dual terms as
$$
\rho(x,Q_n)_1 = \sup_{z\in K} \langle x,z\rangle,
\quad K := B_\infty^N \cap Q_n^\perp.
$$
So, the central sections of the cube $B_\infty^N$ play an important role here. We will
prove the following Lemma.
\begin{lemma}
    \label{lem_cube_section}
    Let $K$ be an $m$-dimensional central section of the cube $B_\infty^N$, and $m\ge \eps
    N$. Then $\vc(K,c_1(\eps)) \ge c_2(\eps)N$.
\end{lemma}

The proof of this Lemma relies on two useful results.
\begin{exttheorem}[\cite{V79}]
    \label{ext_vaaler}
    Any $m$-dimensional section of a volume one cube has volume at least
    one: $\Vol_m(\frac12 B_\infty^N \cap L_m)\ge 1$.
\end{exttheorem}
In fact, a weaker bound $\Vol_m(B_\infty^N\cap L_m)\ge c^m$
will suffice for our purposes.

\begin{exttheorem}[\cite{MV03}]
    Let $\mathcal F$ be a class of functions bounded by $1$, defined on a set
    $I$. Then for any probability measure $\mu$ on $I$,
    \begin{equation}
        \label{mv_bound}
    \log N_t(\mathcal F,L_2(\mu)) \le C\cdot \vc(\mathcal F,ct)\log(2/t),\quad 0<t<1,
    \end{equation}
    where $C$ and $c$ are positive absolute constants.
\end{exttheorem}
Here $N_\delta(K,X)$ is the size of the minimal $\delta$-net
for a set $K$ in the space $X$.

\begin{proof}[Proof of Lemma~\ref{lem_cube_section}]
Let $K=L_m\cap B_\infty^N$. From Theorem~\ref{ext_vaaler} we get that $\Vol_m(K)\ge 2^m$ and we can use standard volume estimate
$$
    N_r(K,\ell_2^N) = N_r(K,\ell_2^N\cap L_m) \ge
    \frac{\Vol_m(K)}{\Vol_m(r B_2^m)} \ge \frac{2^m}{(cr m^{-1/2})^m}\ge
    2^m,\quad\mbox{for $r=c_1m^{1/2}$}.
$$
    We will apply~\eqref{mv_bound} to the set $K$ (note that $\|z\|_\infty\le 1$
    for $z\in K$ by definition) in the space $L_2^N$
     for $t=N^{-1/2}r\asymp \eps^{1/2}$:
$$
\vc(K,ct)\ge c_1\log N_t(K,L_2^N)/\log(2/t) \ge c(\eps)N.
$$
\end{proof}

Now we are ready to prove Theorem~\ref{th_l1}.
\begin{proof}
    By duality, we have
$$
\E\rho(\xi,Q_n)_1 = \E \max_{z\in K} \langle z,\xi\rangle,\quad K := B_\infty^N\cap Q_n^\perp.
$$

    Lemma~\ref{lem_cube_section} gives a set of coordinates
    $\Lambda\subset\{1,\ldots,N\}$, $|\Lambda|\ge c(\eps)N$, and a number $v\ge
    c(\eps)$, such that for any choice of signs $\sigma\colon\Lambda\to\{-1,1\}$
there is a vector $z^\sigma\in K$ with
    \begin{equation}
        \label{lambda_final}
        \min_{i\in\Lambda} \sigma_i z^\sigma_i \ge v.
    \end{equation}
    (Recall that as $K$ is convex and centrally-symmetric we may assume that
    $z^\sigma$ oscillate around zero.)

    We will derive from~\eqref{lambda_final} that
    \begin{equation}
        \label{sharp}
    \E \max_\sigma \langle \xi,z^\sigma \rangle \ge v \sum_{i\in\Lambda} \E|\xi_i| -
        \sum_{i\not\in\Lambda} \|\E(\xi_i|\{\sign\xi_j\}_{j\in\Lambda})\|_{L_1}.
    \end{equation}

    Let us fix
    $\sigma^*\colon\Lambda\to\{-1,1\}$ and let $\E^*$ be the conditional expectation
    $\E^*(\cdot)=\E(\cdot|\sign\xi_i=\sigma_i^*,\,i\in\Lambda)$. We have
    $$
    \E^*\max_\sigma\langle\xi,z^\sigma\rangle \ge
    \E^*\langle\xi,z^{\sigma^*}\rangle =
    \E^*\sum_{i\in\Lambda}\xi_i z^{\sigma^*}_i + \sum_{i\not\in\Lambda}\xi_i
    z^{\sigma^*}_i \ge
    v\sum_{i\in\Lambda}\E^*|\xi_i| - \sum_{i\not\in\Lambda}|\E^*\xi_i|.
    $$
    Here we used that $\|z^{\sigma^*}\|_\infty\le 1$. We average over $\sigma^*$
    and obtain~\eqref{sharp}.

    The Theorem follows from~\eqref{sharp}, because the first sum there is at
    least $v|\Lambda|\ge c_\eps N$, and for each term in the second sum we have
    $$
    \|\E(\xi_i|\{\sign\xi_j\}_{j\in\Lambda})\|_{L_1} \le
    \|\E(\xi_i|\{\sign\xi_j\}_{j\ne i})\|_{L_1}.
    $$
\end{proof}

Let us give some corollaries of Theorem~\ref{th_l1} for the case when
$\E(\xi_i|\{\sign\xi_j\}_{j\ne i})\equiv0$.

\begin{corollary}
    \label{cor_indep_l1}
    Let $\xi_1,\ldots,\xi_N$ be independent random variables with $\E\xi_i=0$
    and $\E|\xi_i|\ge 1$. Then for any
    r.v.s $\eta_1,\ldots,\eta_n$, $n\le N(1-\eps)$, and any coefficients $\{a_{i,j}\}$ we have
    $$
    \E\|\xi-A\eta\|_1 = \sum_{i=1}^N \E|\xi_i-\sum_{j=1}^n a_{i,j}\eta_j| \ge
    c(\eps)N.
    $$
    In terms of widths, we have
    $$
    d_n^\avg(\{\xi_1,\ldots,\xi_N\},L_1) = N^{-1}d_n^\avg(\xi,\ell_1^N) \ge c(\eps).
    $$
    This also holds true for unconditional vectors $\xi$, that is, if
     $\Law(\xi_1,\ldots,\xi_N)=\Law(\pm\xi_1,\ldots,\pm\xi_N)$.
\end{corollary}

\begin{example}
    Suppose that $\int_0^1 f_k(x)\,dx=0$, $\int_0^1 |f_k(x)|\,dx=1$ for
    $k=1,\ldots,N$. Then
    $$
    d_n^\avg(\{f_1(x_1),\ldots,f_N(x_N)\}, L_1[0,1]^N)\ge c(\eps),
    \quad\mbox{if $n\le N(1-\eps)$}.
$$
\end{example}

\subsection{Rigidity in $L_0$}
\label{subsec_l0}

\begin{theorem}
    \label{th_l0}
    For any $\eps\in(0,1)$ there exists $\delta>0$, such that if a random vector
    $\xi=(\xi_1,\ldots,\xi_N)$ has independent coordinates and
    $$
    \inf_{c\in\R}\|\xi_i-c\|_{L_0}\ge\eps,\quad i=1,\ldots,N,
    $$
    then for any subspace $Q_n\subset\R^N$ of dimension $n\le \delta N$, we have
    \begin{equation}\label{exp_prob}
        \P(\rho(\xi,Q_n)_{L_0^N} \le \delta)\le 2\exp(-\delta N),
    \end{equation}
    where $\xi=(\xi_1,\ldots,\xi_N)$.
\end{theorem}

The condition $\|\xi_i-c\|_{L_0}\ge\eps$ is essential here because it
forbids the approximation of the vector $\xi$ by a constant vector. This
condition allows to somehow ``separate'' the values of $\xi_i$.

\begin{lemma}\label{lem_l0}
    Let $\eps,\tau\in(0,1]$ and $\zeta$ is a r.v. such that
    $\inf_{c\in\R}\|\zeta-c\|_{L_0}\ge\eps$. Then there exist real numbers $a<b$, such that
    $\P(\zeta\le a)\ge \eps/3$, $\P(\zeta\ge b)\ge \eps/3$,
    $\P(a<\zeta<b)\le\tau$,
    $|b-a|\ge\tau\eps/2$.
\end{lemma}

\begin{proof}
    Consider the quantiles 
    $$
    q_- := \inf\{x\colon \P(\zeta\le x)\ge \eps/3\},\quad
    q_+ := \sup\{x\colon \P(\zeta\ge x)\ge \eps/3\}.
    $$
    Note that $\P(\zeta\le q_-)\ge \eps/3$ and $\P(\zeta\ge q_+)\ge \eps/3$.
    We see that $|q_+-q_-|\ge\eps$, because otherwise $\|\zeta-c\|_{L_0}\le\frac23\eps$ for
    $c=\frac12(q_-+q_+)$. Pick $k:=\lceil1/\tau\rceil$ and divide
    $[q_-,q_+]$ into $k$ equal segments and choose $(a,b)$ to be the segment
    having the least probability $\P(a<\zeta<b)$. Then this probability is at most $1/k\le\tau$.
    Finally, $|b-a|\ge |q_+-q_-|/k \ge \frac12\tau\eps$.
\end{proof}

Let us prove the Theorem.
\begin{proof}
    Pick some small $\delta>0$ (to be chosen later) and estimate
    $\P(\rho(\xi,Q_n)_{L_0^N}<\delta)$, where $\dim Q_n\le n\le \delta N$.
    Consider an almost best $L_0^N$ approximation of 
    $(\xi_1,\ldots,\xi_N)$ by a (random) vector $(y_1,\ldots,y_N)$ from the
    subspace $Q_n$.
    Let
    $$
    \Lambda := \{k\colon |\xi_k-y_k| \ge \delta\}
    $$
    be the (random) set of badly approximated coordinates. If
    $\|\xi-y\|_{L_0^N} < \delta$, then $|\Lambda|<\delta N$; so it is enough to bound the
    probability of the event $\{|\Lambda| \le \delta N\}$.

    We apply Lemma~\ref{lem_l0} to $\tau := 5\delta/\eps$
    and $\xi_k$ and find segments $(a_k,b_k)$ such that:
    $\P(a_k<\xi_k<b_k)\le \tau$, $|b_k-a_k|\ge\tau\eps/2 = 5\delta/2$,
    $\max\{\P(\xi_k\le a_k),\P(\xi_k\ge b_k)\}\le1-\eps/3$.
    Let
    $$
    \Gamma := \{k\colon \xi_k\in(a_k,b_k)\}
    $$
    be the (random) set of ``intermediate'' coordinates.  We have $\E|\Gamma|\le
    \tau N$; by Bernstein inequality we have $\P(|\Gamma|\le 2\tau
    N)\ge 1-2\exp(-c\tau N)$ and
    \begin{equation}
        \label{prob_lambda}
    \P(|\Lambda|\le \delta N) \le \P(|\Lambda|\le \delta N,\;|\Gamma|\le 2\tau N) + 2\exp(-c\tau N).
    \end{equation}
    
    Fix some sets $\Lambda^\circ,\Gamma^\circ\subset\{1,\ldots,N\}$ of size
    $|\Lambda^\circ|\le\delta N$, $|\Gamma^\circ|\le 2\tau N$ and consider the event
    $\{\Lambda=\Lambda^\circ,\;\Gamma=\Gamma^\circ\}$. 
    Let $N'=N-|\Lambda^\circ\cup\Gamma^\circ|$; note that $N'\ge N/2$. 
    For $x\in\R^N$ let $x'\in\R^{N'}$ be the vector $x$ without coordinates from
    $\Lambda^\circ\cup\Gamma^\circ$. Then
    $\|\xi'-y'\|_{\ell_\infty^{N'}} \le \delta$ with $y'\in Q_n'$; therefore,
    \begin{equation}
        \label{d_K}
    d_n(K,\ell_\infty^{N'})\le\delta,\quad
        K:=\{\xi'(\omega)\colon \Lambda(\omega)=\Lambda^\circ,\;\Gamma(\omega)=\Gamma^\circ\}.
    \end{equation}
    Consider the random vector $\pi$ with coordinates $\pi_k := -1$ if $\xi_k\le
    a_k$; $\pi_k := 0$ if $a_k<\xi_k<b_k$; $\pi_k := 1$ if $\xi_k\ge b_k$.
    Note that $\pi'\in\{-1,1\}^{N'}$ by construction. We argue that the set
    $S := \{\pi'(\omega)\colon
    \Lambda(\omega)=\Lambda^\circ,\;\Gamma(\omega)=\Gamma^\circ\}$
    of all possible values of $\pi'$ has at most $(eN'/n)^{n}$ elements.
    Indeed, otherwise~\eqref{vc_std},~\eqref{binom} imply that $S$ contains a cube of
    dimension $n+1$ and
    $$
    \vc(K,t)\ge n+1,\quad t:=\frac12\min(b_k-a_k)\ge 5\delta/4,
    $$
    so~\eqref{vc_width} implies that $d_n(K,\ell_\infty^{N'})\ge 5\delta/4$, that
    contradicts~\eqref{d_K}. Given any $s\in S$ we have
    $\P(\pi'=s)\le(1-\eps/3)^{N'}$, so
    $$
    \P(\Lambda=\Lambda^\circ,\;\Gamma=\Gamma^\circ) \le
    (eN/n)^n(1-\eps/3)^{N/2}.
    $$
    Finally, we bound the probability in~\eqref{prob_lambda}:
    \begin{multline*}
        \P(|\Lambda|\le \delta N,\;|\Gamma|\le 2\tau N) =
        \sum_{|\Lambda^\circ|\le\delta N,|\Lambda^\circ|\le 2\tau N}\P(\Lambda=\Lambda^\circ,\;\Gamma=\Gamma^\circ) \le \\ 
        \le (eN/(\delta N))^{\delta N}\cdot (eN/(2\tau N))^{2\tau N} \cdot (eN/n)^{n}\cdot
        (1-\eps/3)^{N/2}.
    \end{multline*}
    It remains to choose the parameter $\delta$ small enough to
    ensure that this probability is exponentially small.
\end{proof}

We can give a corollary on the best $n$-term approximation by a dictionary:
$$
\sigma_n(x,\Phi)_X := \inf_{\substack{\varphi_1,\ldots,\varphi_n\in\Phi\\c_1,\ldots,c_n\in\R}}
\|x-\sum_{k=1}^n
        c_k\varphi_k\|_X.
$$

\begin{corollary}
    In the conditions of Theorem~\ref{th_l0}, for any set $\Phi\subset\R^N$,
    $|\Phi|\le AN$ and $n<\delta N$, we have
    $$
        \P(\sigma_n(\xi,\Phi)_{L_0^N} \le \delta)\le 2\exp(-\delta N)
    $$
\end{corollary}
\begin{proof}
Indeed, there are exponential number of subspaces, so we can choose smaller
    $\delta=\delta(A,\eps)$ to provide an exponential bound on the probability.
\end{proof}

We need a modification of Statement~\ref{stm_equiv}.
\begin{lemma}
    \label{lem_equiv_l0}
    For any random vector $\xi$ in $\R^N$ we have
    $$
    c\,d_n^\avg(\xi,L_0^N)^4 \le d_n^\avg(\{\xi_1,\ldots,\xi_N\},L_0)
    \le C\,d_n^\avg(\xi,L_0^N)^{1/4}.
    $$
\end{lemma}

\begin{proof}
    Let $d_n^\avg(\{\xi_1,\ldots,\xi_N\},L_0) \le \eps\le 1$. Then for some
    $\eta_1,\ldots,\eta_N$ from an $n$-dimensional space we have
    $\sum\|\xi_k-\eta_k\|_{L_0}\le N\eps$. Define random sets
    $\Lambda_t := \{k\colon |\xi_k-\eta_k|\ge t\}$; note that $\|\xi-\eta\|\ge
    t$ if and only if $|\Lambda_t|\ge tN$.

    Take $\delta := \sqrt\eps$ and obtain
    \begin{multline*}
    \E|\Lambda_\delta| = \sum_{k=1}^N \P(|\xi_k-\eta_k|\ge\delta) \le \delta N +
    |\{k\colon\P(|\xi_k-\eta_k|\ge\delta)\ge\delta\}| \le \\
        \le \delta N + |\{k\colon \|\xi_k-\eta_k\|_{L_0}\ge\delta\}| \le \delta N +
        \eps N / \delta = 2\delta N.
    \end{multline*}
    Moreover, for $\gamma := \sqrt{2\delta}$ we have
    $$
    \P(\|\xi-\eta\|_{L_0^N}\ge\gamma) \le \P(|\Lambda_\gamma|\ge \gamma N) \le
    \frac{\E|\Lambda_\gamma|}{\gamma N} \le \frac{\E|\Lambda_\delta|}{\gamma N}
    \le \frac{2\delta N}{\gamma N} = \gamma.
    $$
    Hence $\E\|\xi-\eta\|_{L_0^N} \le 2\gamma$.

    The proof of the second inequality is analogous.
\end{proof}

\begin{corollary}
    \label{cor_indep_l0}
    In the conditions of Theorem~\ref{th_l0}, we have
    $$
    d_n^{\mathrm{avg}}(\{\xi_1,\ldots,\xi_N\},L_0)\ge\delta\quad\mbox{if $n\le
    \delta N$.}
    $$
\end{corollary}

\subsection{Random sets}

Let $A$ be a signum matrix, i.e. a matrix with $\pm1$ entries. Define the
signum rank $\rank_\pm A$ is the minimal rank of matrices $B$ with $\sign
B_{i,j}\equiv A_{i,j}$ ($B_{i,j}\ne0$). 
There is is a bound on the number of $N\times N$ signum matrices with
low signum rank via~\cite{AFR85}, see also~\cite{AMY16,Mal22}:
$$
|\{A\in\{-1,1\}^{N\times N}\colon \rank_\pm A \le r\}| \le (4eN/r)^{2rN}.
$$
For $r=cN$ we obtain the bound $2^{bN^2}$, where $b=b(c)\to0$ if $c\to0$.
Hence if $c$ is small there are very few signum matrices such that
$\rank_\pm A\le cN$. Moreover, almost all signum matrices are far away from such
matrics in Hamming metric. Hence almost all signum matrics are far away from
low-rank matrices, $\rank B\le cN$, in the $L_0$ distance.
We obtain the following.

\begin{statement}
    Let $\mathbf{\mathcal E}$ be a random matrix with independent $\pm1$
    entries. Then
    $$
    \P(\min_{\rank B\le c N}\|\mathbf{\mathcal E}-B\|_{L_0^{N\times
    N}}\le c)\le 2\exp(-c N^2).
    $$
\end{statement}

Let $D_N(a,b)$ be the subset of $L(a,b)$ of functions that are piece-wise
constant on intervals $(a + (b-a)\frac{j-1}{N},a + (b-a)\frac{j}N)$,
$j=1,\ldots,N$. From the previous Statement we obtain:

\begin{statement}
    Let $\mathbf{\varphi}_1,\ldots,\mathbf{\varphi}_N$ be a random system in
    $D_N(0,1)$,
    $\left.\varphi_k\right|_{(\frac{j-1}N,\frac{j}N)}=\eps_{k,j}$, where
    $\eps_{k,j}$ are independent random signs. Then with probability at least
    $1-2\exp(-cN^2)$ we have
    $$
    d^\avg_n(\{\mathbf{\varphi}_1,\ldots,\mathbf{\varphi}_N\},L_0)\ge
    c,\quad\mbox{if $n\le cN$.}
    $$
\end{statement}

Note that although we use independent variables in the definition of
$\varphi_k$, the system $\{\varphi_k\}_1^N$ is by no means independent.

\begin{proof}
    Let us prove that if $\{\varphi_k\}$ is not rigid in $L_0(0,1)$, then the
    matrix $\Phi := (\varphi_{k,j})$ may be well approximated in
    $L_0^{N\times N}$; this will contradict the previous Statement.

    Consider an approximation
    $\varphi_k\approx g_k$ in $L_0$ by functions in some $n$-dimensional
    subspace. There is a technical subtlety: we can't claim that
    $g_k\in D_N$ (one can't average over segments because such operator is not
    defined in $L_0$).
    Instead we use Lemma~\ref{lem_equiv_l0}: it shows that the rigidity of the
    system $\xi_k := \varphi_k$ is equivalent of the rigidity of the vector
    $(\xi_1,\ldots,\xi_N)$ hence it is determined by the distribution of this
    vector. So we can assume that the measure of the space $(0,1)$ where
    everything is defined, consists of ``atoms'' $((j-1)/N,j/N)$.

    Finally, if $\varphi_k\approx g_k$, $g_k\in D_N$ and the average of
    $\|\varphi_k - g_k\|_{L_0}$ is less than $\delta$, we have
    $\|\varphi_k-g_k\|_{L_0}\le \gamma := \delta^{1/2}$ for all except at most
    $\gamma N$ indices. Therefore, $\|\Phi-G\|_{L_0^{N\times N}} \le
    2\gamma$.
\end{proof}

\begin{corollary}
    There exists a orthonormal system $\varphi_1,\ldots,\varphi_N\in D_N(0,1)$ that is rigid in $L_0$:
    $$
    d_n^{\mathrm{avg}}(\{\varphi_1,\ldots,\varphi_N\}, L_0(0,1)) \ge c,
    \quad\mbox{if $n\le cN$.}
    $$
\end{corollary}

\begin{proof}
    We will construct a system $\varphi_1,\ldots,\varphi_N\in
    D_{2N}(0,2)$ instead.  Define $\varphi_1,\ldots,\varphi_N$ on $(0,1)$ using random values
    $\left.\varphi_i\right|_{(\frac{j-1}N,\frac{j}N)} = \pm \delta$ with some
    small $\delta$ to establish $L_0$-rigidity with high probability. Then
    continue this system on $(0,2)$ to make it orthonormal. Necessary and
    sufficient condition of the existence of such continuation is known (see,
    e.g.,~\cite[Ch.8]{KS}):
    $$
    \max_{|a|=1}\|\sum_{k=1}^N a_k\varphi_k\|_{L_2(0,1)} \le 1.
    $$
    If $\delta$ is small enough, this condition holds true with high
    probability, because spectral norm of a random $\pm1$ matrix is
    $O(N^{1/2})$, see, e.g.~\cite[Ch.4]{HDPBook}.

    Moreover, we can place the functions $\left.\varphi_k\right|_{(1,2)}$ in any
    $N$-dimensional space (keeping $\int_1^2\varphi_k\varphi_l$), e.g., in $D_N(1,2)$.
\end{proof}

\subsection{Lacunary systems}

Let $\lambda>1$. We say that a sequence of positive integers $k_1,\ldots,k_N$ is
$\lambda$-lacunary, if $k_{j+1}/k_j\ge\lambda$ for $j=1,\ldots,N-1$.

\begin{statement}
    \label{stm_lacunary}
    Let $\varphi$ be a Riemann integrable function on $\T$ and $\mu\{x\in\T\colon \varphi(x)\ne c\}>0$
    for any $c\in\R$.
    \begin{enumerate}
        \item There exists $\lambda=\lambda(\varphi)>1$ such
            that for any $\lambda$-lacunary sequence $k_1,\ldots,k_N$ we have
            $$
            d_n^{\mathrm{avg}}(\{\varphi(k_1x),\ldots,\varphi(k_Nx)\},L_0(\T))\ge
            c(\varphi),\quad\mbox{if $n\le c(\varphi) N$}.
            $$
        \item For any $\eps>0$ there exists $\lambda=\lambda(\varphi,\eps)>1$
            such that for any $\lambda$-lacunary sequence $k_1,\ldots,k_N$,
            we have
            $$
            d_n^{\mathrm{avg}}(\{\varphi(k_1x),\ldots,\varphi(k_Nx)\},L_1(\T))\ge
            c(\varphi,\eps),\quad\mbox{if $n\le N(1-\eps)$}.
            $$
    \end{enumerate}
\end{statement}

We will use the following quantitative result.

\begin{exttheorem}[\cite{B}, Theorem 4.3]
    \label{ext_lacun}
    Let $(k_j)$ be an increasing sequence of positive integers.
    Then there exists a sequence $(g_j)$ of measurable functions on $(0,1)$, independent and
    uniformly distributed over $(0,1)$, such that
    $$
    \mu\{x\in[0,1]\colon |\{k_j x\}-g_j(x)|>2k_j/k_{j+1}\}\le 2k_j/k_{j+1},\quad
    j=1,2,\ldots.
    $$
\end{exttheorem}

Now can prove our Statement.
\begin{proof}
    Let us start with the $L_1$ case.
    First of all, we can shift and renorm $\varphi$ to obtain
    $\int_\T\varphi(x)\,dx=0$, $\int_\T|\varphi(x)|\,dx=1$ (this does not
    affect rigidity).
    Given a $\lambda$-lacunary sequence
    $k_1,\ldots,k_N$ with some $\lambda>1$, we apply Theorem~\ref{ext_lacun} to obtain
    functions $g_1,\ldots,g_N$ with $\|\{k_j x\}-g_j(x)\|_{L_0} \le 2/\lambda$.
    
    Put $f_j(x):=\varphi(g_j(x))$. Then the functions $f_j$ are independent,
    $\int f_j=0$, $\int|f_j|=1$.
    Corollary~\ref{cor_indep_l1} implies the rigidity of $\{f_j\}$. To prove
    the rigidity of $\{\varphi(k_jx)\}$ we have to bound
    $\|\varphi(k_jx)-f_j\|_1$:
    $$
    \int_\T|\varphi(k_jx)-f_j(x)|\,dx =
    \int\limits_A|\varphi(k_jx)-\varphi(g_j(x))|\,dx +
    \int\limits_{\T\setminus A}|\varphi(k_jx)-\varphi(g_j(x))|\,dx,
    $$
    where $A=\{x\colon |\{k_j x\}-g_j(x)|>2/\lambda\}$.
    The first integral is at most $2\|\varphi\|_\infty\mu(A) \le
4\|\varphi\|_\infty/\lambda$. The second integral does not exceed
$$
\int_{\mathbb T}\sup_{|x-t|\le 2/\lambda}|\varphi(t)-\varphi(x)|\,dx =:
\tau(\varphi,4/\lambda),
$$
    where $\tau$ is the averaged moduli of continuity. It is known
    that $\lim_{h\to 0}\tau(\varphi,h)=0$ for Riemann-integrable functions.
    Hence, for large enough $\lambda$ the value of
$\|\varphi(k_jx)-f_j(x)\|_1$ will be sufficiently small.
This proves the $L_1$ case.

    The $L_0$ case also follows because $\|\varphi(k_jx)-f_j(x)\|_{L_0} \le
    \|\varphi(k_jx)-f_j(x)\|_1^{1/2}$ and we can use
    Corollary~\ref{cor_indep_l0}.
\end{proof}

\subsection{Complemented subspaces}

In this paragraph we will outline another method for the proof of rigidity, that
is based on the notion of complementability.

Let $f_1,\ldots,f_N$ be an orthonormal system in $L_2(0,1)$. Suppose that we aim
to prove rigidity of $\{f_j\}$ in a larger space $X\supset L_2$.
If there exists a bounded linear operator $\pi\colon X\to L_2$ that keeps the
system: $\pi f_j=f_j$, $j=1,\ldots,N$, then we obtain the rigidity in $X$:
$$
\|\pi\|_{X\to L_2} \cdot d_n^\avg(\{f_1,\ldots,f_N\},X)_2 \ge
d_n^\avg(\{f_1,\ldots,f_N\},L_2)_2 = (1-n/N)^{1/2}.
$$

Let us give a corollary:
\begin{statement}
    \label{stm_chaos}
Let $\{f_j\}_1^\infty$ be an orthonormal system in $X\supset L_2$
and two conditions are satisfied:
\begin{enumerate}
    \item[(1)] the system spans an $\ell_2$ subspace:
        $\|\sum_{j\ge 1}c_jf_j\|_X\asymp \|\sum_{j\ge1}c_jf_j\|_2$,
    \item[(2)] the system in complemented, i.e. there is a bounded projector $\pi\colon
        X\twoheadrightarrow F_X := \overline{\mathrm{span}}\,\{f_j\}_{j=1}^\infty$.
\end{enumerate}
    Then
    $$
    d_n^\avg(\{f_{i_1},\ldots,f_{i_N}\},X)_2\ge c\sqrt{1-n/N},\quad\mbox{where
    $c=c(F_X,X)>0$},
    $$
    for any $i_1<i_2<\ldots<i_N$.
\end{statement}

\begin{corollary}
    The Rademacher chaos $\{r_ir_j\}_{1\le i<j<\infty}$ is rigid in $L\log L$:
    for any set of pairs $\Lambda$ we have
    $$
    d_n^\avg(\{r_ir_j\}_{(i,j)\in\Lambda}, L\log L)_2 \ge c(\eps)>0\quad\mbox{if
    $n<|\Lambda|(1-\eps)$.}
    $$
\end{corollary}

\begin{proof}
    Let $X$ be some symmetric space. Is is known that property (1) is equivalent
    to the embedding $X\supset H$, see.~\cite[Theorem
    6.4]{Ast}; the space $H$ is the closure of $L_\infty$
    in the Orlicz space $L_{\varphi_1}$, $\varphi_1(t)=e^t-1$.
    The criterium for complementability is that $H\subset
    X\subset H'$, see~\cite[Theorem 6.7]{Ast}. Hence for $X=H'$
    both conditions are satisfied. The space $H'$ coincides with the Lorentz
    space $\Lambda(\psi_1)$ of functions with $\int_0^1 f^*(t)\log(2/t)\,dt<\infty$
    (see~\cite[\S2.2,\S2.4]{Ast}). Further, this Lorentz space coincides~(see,
    e.g.,~\cite[Theorem D]{BR}) with the space $L\log L$: $\int_0^1 |f|\log(2+|f|)\,dt<\infty$.
\end{proof}

\section{Rigidity in $L_p$, $1<p<2$}
\label{sec_lp}

\subsection{$S_{p'}$-property}
\label{subsec_lp12}

We start with a finite-dimensional result. It's proof follows the
paper~\cite{G87} of E.D.~Gluskin.

\begin{theorem}
    \label{th_p12}
    Let $1<p<2$ and $N\in\mathbb N$. Suppose that
    $\xi=(\xi_1,\ldots,\xi_N)$ is an isotropic random vector in $\R^N$:
    $\E\xi_i^2=1$, $\E\xi_i\xi_j=0$ ($i\ne j$), that satisfies two conditions:
    \begin{equation}
        \label{A_condition}
        (\E|\xi|^{2+\gamma})^{\frac1{2+\gamma}} \le AN^{\frac12}\quad\mbox{for
        some $\gamma>0$, $A\ge 1$,}
    \end{equation}
    \begin{equation}
        \label{B_condition}
        \max_{|v|=1}(\E|\langle \xi,v\rangle|^{p'})^{\frac1{p'}} \le
        B\quad\mbox{for some $B\ge1$}.
    \end{equation}
    Then for any $\eps\in(0,1)$, $n\le N(1-\eps)$ and
    any $n$-dimensional space $Q_n\subset\R^N$ we have
    $$
    \P\{\rho(\xi,Q_n)_{\ell_p^N}\ge c(A,\eps,\gamma)N^{1/p} B^{-1}\} \ge
    c(A,\eps,\gamma).
    $$
\end{theorem}

\begin{proof}
    Let $P$ be the orthoprojector on the space $Q_n^\perp$.
    \begin{equation}
        \label{rho}
    \rho(\xi,Q_n)_p = \max_{z\in Q_n^\perp} \frac{\langle \xi,z\rangle}{\|z\|_{p'}}
    \ge \frac{\langle \xi,P\xi\rangle}{\|P\xi\|_{p'}}.
    \end{equation}

    First we analyze $\eta := \langle\xi,P\xi\rangle$, the numerator of~\eqref{rho}:
    $$
    \E\eta = \E\langle\sum\xi_i e_i,\sum \xi_i Pe_i\rangle =
    \sum \langle e_i,Pe_i\rangle = \tr P = N -n \ge \eps N.
    $$
    On the other side,
    $$
    \E\eta \le \frac12\eps N + \E\eta\{\eta\ge\frac12\eps N\} \le
        \frac12\eps N +
        (\E|\eta|^{1+\gamma/2})^{\frac2{2+\gamma}}\P(\eta\ge\frac12\eps
        N)^{\frac{\gamma}{2+\gamma}}.
    $$
    As $|\eta|=|\langle\xi,P\xi\rangle|\le|\xi|^2$, we have
    $\E|\eta|^{1+\gamma/2}\le (AN^{1/2})^{2+\gamma}$
    and we conclude that
    \begin{equation}
        \label{prob_numer}
        \P(\eta \ge \frac12\eps N)\ge t :=
        \left(\frac{\eps}{2A^2}\right)^{\frac{2+\gamma}\gamma}.
    \end{equation}

    Now, the denominator of~\eqref{rho}:
    $$
    \E\|P\xi\|_{p'}^{p'} = \sum_{k=1}^N \E|\langle P\xi,e_k\rangle|^{p'} = 
    \sum_{k=1}^N \E|\langle \xi,Pe_k\rangle|^{p'} \le NB^{p'}.
    $$
    Using Markov's inequality we obtain that
    \begin{equation}
        \label{prob_denom}
        \P(\|P\xi\|_{p'}^{p'} \ge 2t^{-1} NB^{p'})\le \frac12t.
    \end{equation}
    
    We combine~\eqref{prob_numer} and~\eqref{prob_denom} to get the required bound
    $$
    \rho(\xi,Q_n)_p \ge \frac{\langle \xi,P\xi\rangle}{\|P\xi\|_{p'}} \ge
    \frac{\frac12\eps N}{2^{1/p'}t^{-1/p'}N^{1/p'}B} \ge \frac14 \eps t^{1/2}
    N^{1/p}B^{-1}
    $$
    with probability at least $t/2$.
\end{proof}

\begin{remark}
    Note that if the condition~\eqref{B_condition} is satisfied,
    then~\eqref{A_condition} also holds true with $A:=B$, $\gamma:=p'-2$,
    see the proof of Theorem~\ref{th_p12_func}.
    However, is some cases we need the exact dependence $(\cdots)B^{-1}$ of the lower
    bound on the parameter $B$.
\end{remark}

Let us formulate a corollary in function-theoretic terms.

\begin{theorem}
    \label{th_p12_func}
    Let $\varphi_1,\ldots,\varphi_N$ be an orthonormal system in a space
    $L_2(\mathcal X,\mu)$, $\mu(\mathcal X)=1$.  Let $1<p<2$ and suppose that $\{\varphi_k\}$ has
    $S_{p'}$-property, i.e. for some $B\ge1$ one has
    $$
    \|\sum_{k=1}^N a_k\varphi_k \|_{p'} \le B|a|,\quad\forall a\in\R^N.
    $$
    Then this system is rigid in $L_p$:
    $$
    d_n^{\mathrm{avg}}(\{\varphi_1,\ldots,\varphi_N\},L_p)_p \ge
    c(B,\eps,p),\quad\mbox{if $n\le N(1-\eps)$}.
    $$
\end{theorem}

\begin{proof}
Indeed, let $x$ be a random point in $\mathcal X$ and consider the random vector
$\xi=(\varphi_1(x),\ldots,\varphi_N(x))$. 
    The condition~\eqref{B_condition} of Theorem~\ref{th_p12} is exactly the $S_{p'}$--property.
    Let us check that the condition~\eqref{A_condition}  is fullfilled with
    $A:=B$ and $\gamma=p'-2$.
    Indeed, 
    $$
    \|\varphi_k^2\|_{1+\gamma/2} = \|\varphi_k\|_{2+\gamma}^2 = \|\varphi_k\|_{p'}^2 \le B^2,
    $$
    $$
    \|(\sum_{k=1}^N \varphi_k^2)^{1/2}\|_{2+\gamma}^2 = \|\varphi_1^2+\ldots+\varphi_N^2\|_{1+\gamma/2} \le
    \sum_{k=1}^N \|\varphi_k^2\|_{1+\gamma/2} \le B^2N.
    $$
    Hence $d_n^{\mathrm{avg}}(\xi,\ell_p^N)_p$ is bounded from below.
    It remains to use Statement~\ref{stm_equiv}.
\end{proof}

\begin{corollary}
    Let $\varphi_1,\ldots,\varphi_N$ be an orthonormal system in $L_2(\mathcal X,\mu)$,
    $\mu(\mathcal X)=1$, that
    is uniformely bounded: $\|\varphi_k\|_\infty\le A$. Then for $1<p<2$ we have
    $$
    d_n(\{\varphi_1,\ldots,\varphi_N\},L_p) \ge
    c(A,p)>0,\quad\mbox{if $n\le \frac12 N^{2/p'}$}.
    $$
\end{corollary}
\begin{proof}
    Indeed, by a Theorem of Bourgain~\cite{B89}, one can find a subsystem of
$\{\varphi_k\}$ of size $\lfloor N^{2/p'}\rfloor$ with $S_{p'}$-property and
apply previous Corollary to that subsystem.
\end{proof}

\subsection{Additional remarks}

\paragraph{Explicit rigid sets.}

Previous considerations allows us to give explicit constructions of $\ell_p$-rigid
subsets of $\{-1,1\}^N$ of polynomial size.
\begin{statement}
    \label{stm_explicit_rigid}
    For any $p\in(1,2)$ and any sufficiently large $N$ one can explicitly
    construct a set $V\subset\{-1,1\}^N$ of size $|V|\le N^{C(p)}$ such that
    $d_{N/2}(W,\ell_p^N)\ge c(p) N^{1/p}$.
\end{statement}

\begin{proof}
    It is sufficient to construct $N$ characters $\chi_1,\ldots,\chi_N\in \widehat{\mathbb
    Z_2^k}$ (with appropriate $k$) that form an $S_{p'}$-system.
    Then we apply Theorem~\ref{th_p12_func}:
    $d_{N/2}^\avg(\{\chi_1,\ldots,\chi_N\},L_p)_p\ge c>0$. Hence once can
    take $V:=\{(\chi_1(x),\ldots,\chi_N(x))\colon x\in \mathbb
    Z_2^k\}$.

    In the paper~\cite{Haj86} the $S_{2m}$-system of $2^n$ characters in $\mathbb
    Z_2^{mn}$ is constructed, so it remains to take $2^n\approx N$, $m=\lceil p'/2\rceil$,
    $k=mn$.
\end{proof}

\paragraph{Independence in $L_p$.}
The analog of Theorem~\ref{th_l1} for independent random variables with
$L_p$-normalization holds true for $1<p\le 2$. One can derive the rigidity in
this case using Theorem~\ref{th_p12} (up to a logarithmic factor), but this
question is the subject of another paper.
However, the rigidity does not hold true for $p>2$.

\begin{statement}
    For any $p\in(2,\infty)$ there exists $\delta>0$, such that for all sufficiently large $N$ there is
    a sequence $\xi_1,\ldots,\xi_N$ of independent identically distributed
    random variables having $\E\xi_i=0$, $\E|\xi_i|^p=1$, and
    $$
    d_{N^{1-\delta}}(\{\xi_1,\ldots,\xi_N\},L_p) \le C(p)N^{-\delta}.
    $$
\end{statement}

\begin{proof}
    We will use the approximation of octahedron $B_1^N$ in $\ell_\infty^N$
    (see~\cite{K75}). Let $n=N^{1-\beta}$, $\beta$ to be
    chosen later, and $Q_n^*$ be the extremal $n$-dimensional
    $\ell_\infty$-subspace:
    $$
    \rho(B_1^N,Q_n^*)_\infty = d_n(B_1^N,\ell_\infty^N) \le C(\beta) n^{-1/2}.
    $$
    We will approximately ``simulate'' vertices of $B_1^N$.
    Pick some small $\eps>0$, and let $\xi_1$ be the random variable s.t.
    $$
    \P(\xi_1=0)=1-\eps,\;\P(\xi_1=K)=\P(\xi_1=-K)=\eps/2,
    $$
    where $K$ is defined by the condition $\E|\xi_1|^p=1$; so, $\eps K^p=1$. Consider independent
    copies $\xi_2,\ldots,\xi_N$ of $\xi_1$ and the vector
    $\xi=(\xi_1,\ldots,\xi_N)$. We will construct an approximation
    vector $\eta$ that lies in $Q_n^*$ (hence $\eta_1,\ldots,\eta_N$ lie in an
    $n$-dimensional subspace of $L^p$).

    Consider three events: $\mathcal{A}_0=\{\xi=0\}$, $\mathcal{A}_1$~--- exactly one coordinate of
    $\xi$ is nonzero; and $\mathcal{A}_{2}$~--- at least two coordinates of $\xi$ are
    nonzero. Define $\eta(\omega):=0$ for $\omega\in \mathcal{A}_0\cup
    \mathcal{A}_{2}$ and let $\eta(\omega)$ be the best approximation of
    $\xi(\omega)$ from $Q_n^*$ if $\omega\in \mathcal{A}_1$.
    In the latter case we have $\|\xi-\eta\|_\infty \le
    K\rho(B_1^N,Q_n^*)_\infty$.

    By symmetry, it is sufficient to estimate $\E|\xi_1-\eta_1|^p$:
    $$
    \E|\xi_1-\eta_1|^p \le \P(\mathcal{A}_1) K^p\rho(B_1^N,Q_n^*)_\infty^p +
    \P(\mathcal{A}_2) K^p.
    $$
    The first term is bounded by
    $$
    N\eps K^p C(\beta)^p n^{-p/2} \ll N^{1-(1-\beta)p/2} \ll N^{-\delta}
    $$
    if $\beta$ is small enough.  The second term is bounded by $N^2\eps^2K^p =
    N^2\eps$ and we can make it small letting $\eps$ to zero.
\end{proof}

\section{Approximation of Walsh and trigonometric functions}
\label{sec_positive}

Let us recall some notations. Given a finite abelian group $G$, we denote by
$\widehat{G}$ the group of charactegs $\chi\colon G\to\mathbb{C}$. The
probabilistic measure $\mu_G$ on $G$ ($\mu_G(A)=|A|/|G|$) allows to define
spaces $\LH$ and $L_0$ of functions on $G$ with Hamming and Ky--Fan metric,
correspondingly (see~\S\ref{subsect_metrics}).

We are interested in two particular cases: the group $\mathbb Z_2^k$ and the
cyclic group $\mathbb Z_N$.

In the first case the characters are real and may be
written as $W_x\colon y\mapsto (-1)^{\langle x,y\rangle}$ ($x,y\in\mathbb Z_2^k$).
We can identify these characters with the functions
of the well--known orthonormal Walsh system (in the Paley numeration)
$w_0\equiv1,w_1,w_2,\ldots$. Indeed, for $n<2^k$ we have
$w_n(\sum_{j\ge 1}y_j2^{-j}) = W_x(y)$,
where
$x=(x_1,\ldots,x_k)\in\{0,1\}^k$, $y=(y_1,\ldots,y_k)\in\{0,1\}^k$,
$n=\sum_{j=1}^kx_j2^{j-1}$.

The second case leads, of course, to discrete Fourier system $y\mapsto e(xy/N)$,
where $e(\theta):=\exp(2\pi i\theta)$.

\subsection{Kolmogorov widths}

Recall the following well-known statement (see, e.g.,~\cite[\S2.1]{TaoWuBook}).

\begin{lemma}
    Let $G$ be a finite Abelian group, $A,B\subset G$ and $|A|+|B|>|G|$. Then
    $A+B=\{a+b\colon a\in A,\,b\in B\} = G$.
\end{lemma}

We will use its immediate consequence.

\begin{lemma}
    \label{lem_half}
    Let $A,B\subset\widehat{G}$ and $|A|+|B|>|G|$. Then for any $m,n\in\mathbb
    N$ we have
    $$
    d_{mn}(\widehat{G},\LH) \le d_m(A,\LH)+d_n(B,\LH).
    $$
\end{lemma}

The next theorem is proven using methos from~\cite{AW17}. It is a modification
of Theorem 1.2 from that paper. 

\begin{theorem}
    Let $k\in\mathbb N$. For any $1\le \lambda\le \frac14\sqrt{k}$ we have
    $$
    d_n(\widehat{\mathbb Z_2^k},\LH) \le 4\exp(-2\lambda^2),\quad
    \mbox{if $n\ge (k/\lambda^2)^{4\lambda\sqrt{k}}$.}
    $$
\end{theorem}

\begin{proof}
    We identify $\mathbb Z_2^k\equiv\{0,1\}^k$ and denote characters as $W_x$:
    $$
    W_x\colon y\mapsto (-1)^{\langle x,y\rangle},\quad x,y\in\{0,1\}^k.
    $$
    Our goal is to approximate $\{W_x\}$ in $\LH$.
    Let $\mathcal X$ be the set of $x$ such that $|\sum_{i=1}^k
    x_i-k/2|\le\sqrt{k}$. Hoeffding inequality shows that $|X|>2^{k-1}$, so it
    is enough to approximate $W_x$ for $x\in\mathcal X$ (and the use
    Lemma~\ref{lem_half}).

    Fix $x\in\mathcal X$. Let $I:=\{i\colon x_i=1\}$,
    $$
    Y_I := \left\{y\in\{0,1\}^k\colon \left|\sum_{i\in I} y_i -
    \frac{|I|}2\right| \le \lambda|I|^{1/2}\right\}.
    $$
    Then $2^{-k}|Y_I| \ge 1-2\exp(-2\lambda^2)$, and 
    so for any $y\in Y_I$ we have
    \begin{multline*}
    |\langle x,y\rangle-k/4| = |\sum_{i\in I}y_i - k/4| \le|\sum_{i\in
        I}y_i-\frac{|I|}2|+|\frac{|I|}2-k/4| \le \\
    \le\lambda\sqrt{|I|}+\sqrt{k}/2\le 2\lambda\sqrt{k}.
    \end{multline*}
    We consider the polynomial $q(t)$ that interpolates $(-1)^t$ for
    $t\in\mathbb Z\cap[k/4-2\lambda\sqrt{k},k/4+2\lambda\sqrt{k}]$. It's degree
    $d$ is at most $4\lambda\sqrt{k}$; note that $4\lambda\sqrt{k}\le k$.
    Then we have $W_x(y)=q(\langle x,y\rangle)$ for all $y\in Y_I$, hence
    $$
    \|W_x - q(\langle x,\cdot\rangle)\|_{\LH} \le 2\exp(-2\lambda^2).
    $$

    Let us bound the linear dimension of the family of functions $\{q(\langle
    x,\cdot\rangle)\}_x$ (it is enough to consider $x\in\mathcal X$ but it is
    more convenient to work with all $x\in\{0,1\}^k$). The dimension is equal to
    the rank of the matrix $(q(\langle x,y\rangle))_{x,y\in\{0,1\}^k}$, where
    $x$ indexes rows and $y$ indexes columns. The rank is bounded by the number of the
    monomials in the polynomial $q(x_1y_1+\ldots+x_ky_k)$, because each monomial
    generates rank-one matrix. Every monomial is a product of some $x_iy_i$, so
    we can bound total number of them:
    $$
    \rank(q(\langle x,y\rangle)_{x,y\in\{0,1\}^k}) \le
    \binom{k}{0}+\ldots+\binom{k}{d} \le (ek/d)^d \le
    (\sqrt{k}/\lambda)^{4\lambda\sqrt{k}}.
    $$
    To finish the proof it remains to apply Lemma~\ref{lem_half}.
\end{proof}

\begin{corollary}
    \label{cor_walsh}
    Let $w_1,\ldots,w_N,\ldots$ be the Walsh system. Then we have
    $$
    d_n(\{w_1,\ldots,w_N\},\LH[0,1])\le 2\exp(-c\log^{0.3}N),\quad
    \mbox{if $n\ge\exp(C\log^{0.7}N)$.}
    $$
\end{corollary}

\begin{theorem}
    Let $k\in\mathbb N$. For any $1\le \lambda \le c_0k^{1/4}$ we have
    $$
    d_n(\widehat{\mathbb Z_{2^k}},L_0) \le c_1k\exp(-c_2\lambda^2),\quad
    \mbox{if $n\ge (k^2/\lambda)^{c_3\lambda^3\sqrt{k}}$}.
    $$
\end{theorem}

\begin{proof}
    We have to approximate characters $y\mapsto e(xy/2^k)$.
    Write $x=\sum_{i=0}^{k-1}x_i2^i$, $y=\sum_{j=0}^{k-1}y_j2^j$ to identify
    $x,y\in\mathbb Z_{2^k}$ with binary vectors
    $(x_0,\ldots,x_{k-1})$, $(y_0,\ldots,y_{k-1})$. Using
    1-periodicity of $e(\cdot)$, we obtain
    $$
    e(xy/2^k)=e(\sum_{i,j=0}^{k-1}x_iy_j2^{i+j-k})=e(\sum_{s=1}^k2^{-s}\sum_{i+j=k-s}x_iy_j).
    $$

    Fix a number $s_0\le\sqrt{k}$ and estimate the tail:
    $$
    \sum_{s>s_0}2^{-s}\sum_{i+j=k-s}x_iy_j\le k\sum_{s>s_0}2^{-s}\le k2^{-s_0}.
    $$
    The function $e(\cdot)$ is $(2\pi)$-Lipshits, so we have a uniform
    approximation
    \begin{equation}
        \label{e_uniform_approx}
        |e(xy/2^k) - \psi_x(y)| \le 2\pi k2^{-s_0},
        \quad \psi_x(y) := e(\sum_{s=1}^{s_0}2^{-s}\sum_{i+j=k-s}x_iy_j).
    \end{equation}
    Hence is enough to approximate functions $\psi_x$.

    As in the previous proof we first deal with
    the set $\mathcal X=\{x\colon |\sum_{i=0}^{k-1}x_i-k/2|\le\sqrt{k}\}$.
    Let us fix $x\in\mathcal X$.

    Fix an index $s\in\{1,\ldots,s_0\}$. We have
    $|\sum_{i=0}^{k-s}x_i-k/2|\le\sqrt{k}+s\le2\sqrt{k}$. 
    Let
    $J_s=\{j\in\{0,\ldots,k-s\}\colon x_{k-s-j}=1\}$, then
    $||J_s|-k/2|\le2\sqrt{k}$ and
    $$
    |\sum_{j\in J_s}y_j-\frac12|J_s||\le\lambda|J_s|^{1/2}
    $$
    for random $y\in\{0,1\}^k$ with probability at least $1-2\exp(-2\lambda^2)$.
    For such $y$ we have
    \begin{multline*}
    \left|\sum_{i+j=k-s}x_iy_j-\frac{k}{4}\right|
    = \left|\sum_{j\in J_s}y_j-\frac{k}{4}\right|
    \le \left|\sum_{j\in J_s}y_j-\frac{|J_s|}{2}\right|+\left|\frac{|J_s|}{2}-\frac{k}{4}\right|\le \\
    \le \lambda\sqrt{\frac{k}{2}+2\sqrt{k}}+\sqrt{k}\le 2\lambda\sqrt{k}.
    \end{multline*}
    (Here we may assume that $k$ is rather large.)

    Let us use the interpolation again. There is a polynomial $q_s$ of degree
    $d\le 4\lambda\sqrt{k}$, that interpolates the function $t\mapsto
    e(2^{-s}t)$ at points $\mathbb Z\cap
    [k/4-2\lambda\sqrt{k},k/4+2\lambda\sqrt{k}]$. Then the function
    $e(2^{-s}\sum_{i+j=k-s}x_iy_j)$ is approximated by $q_s(\langle
    x,y\rangle)$ and the error in the Hamming metric is at most 
    $2\exp(-2\lambda^2)$; wherein $\dim\{q_s(\langle x,\cdot\rangle)\}_x \le (ek/d)^d \le
    (\sqrt{k}/\lambda)^{4\lambda\sqrt{k}}$.

    The function that we wish to approximate is a product:
    \begin{equation}
        \label{e_product}
        \psi_x(y) = \prod_{s=1}^{s_0} e(2^{-s}\sum_{i+j=k-s}x_iy_j).
    \end{equation}
    We approximate it by the product of the corresponding polynomials.
    The error in the Hamming metric is at most
    $2s_0\exp(-2\lambda^2)$. The rank of the product is at most
    product of the ranks, i.e. $(\sqrt{k}/\lambda)^{4s_0\lambda\sqrt{k}}$.

    We use Lemma~\ref{lem_half} to approximate $\psi_x$ in $\LH$
    for all $x$. Finally, we use~\eqref{e_uniform_approx} to obtain the
    $L_0$-approximation of the charaters with the accuracy
    $2\pi k2^{-s_0} + 4s_0\exp(-2\lambda^2)$, using dimension $\le
    (k/\lambda^2)^{4s_0\lambda\sqrt{k}}$.
    It remains to take $s_0\asymp\lambda^2$.
\end{proof}

\begin{corollary}
    \label{cor_trig}
    $$
    d_n(\{e(kx)\}_{k=-N}^N, L_0(\mathbb T))\le 2\exp(-c\log^{0.2}N),\quad
    \mbox{if $n\ge\exp(C\log^{0.9}N)$.}
    $$
\end{corollary}

Here we should first discretize harmonics and obtain DFT matrix of an
appropriate size $2^k\times 2^k$.

\subsection{Trigonometric widths in $L_p$, $p<1$}

Recal that the \textit{trigonometric widths} $d_n^T$ correspond to approximation
by subspaces spanned by functions of trigonometric system.

We will consider two cases alongside: the continuous case with harmonics
$e(kx)$ on $\T$ and the discrete case with harmonics $e(kx/N)$ on $\mathbb Z_N$.
We will use the shorthand $L_p^N$ for the space $L_p(\mathbb Z_N)$ with the
usual (quasi)norm $\|f\|_{L_p^N} = (N^{-1}\sum_{x\in\mathbb Z_N}|f(x)|^p)^{1/p}$.

\begin{theorem}
    \label{thm_trig}
    For any $p\in(0,1)$ and for sufficiently large $N$ we have
\begin{equation}
    \label{trigwidth}
    d_n^T(\{e(kx)\}_{k=-N}^N,L_p) \le \log^{-c_1(p)}N,\quad\mbox{if $n\ge
    N\exp(-\log^{c_2}N)$},
\end{equation}
\begin{equation}
    \label{trigwidth_discrete}
    d_n^T(\{e(kx/N)\}_{k\in \mathbb Z_N}, L_p^N) \le \log^{-c_1(p)}N,\quad\mbox{if $n\ge
    N\exp(-\log^{c_2}N)$}.
\end{equation}
\end{theorem}

Consider the spaces of trigonometric polynomials of degree at most $m$:
$$
\mathcal T_m := \left\{\sum_{k=-m}^m c_k e(kx)\right\},\quad
\mathcal T_m^N := \left\{\sum_{k=-m}^m c_k e(kx/N)\right\}.
$$

\begin{lemma}
    \label{lem_poly}
    For any $p\in(0,1)$ there exists $\delta=\delta(p)>0$, such
    that for all $m,N\in\mathbb N$, $N\ge m$ we have
    $$
    \min\{\|T\|_p\colon T\in\mathcal T_m,\;\widehat{T}(0)=1\} \le C(p)m^{-\delta},
    $$
    $$
    \min\{\|T\|_{L_p^N}\colon T\in\mathcal T_m^N,\;\widehat{T}(0)=1\} \le C(p)m^{-\delta}.
    $$
\end{lemma}

\begin{proof}
    One can take Fejer's kernel and apply the well-known inequality
     $|K_{m-1}(x)|\le \min(m,c/(mx^2))$.
\end{proof}

\begin{lemma}
    \label{lem_set}
    Let $m,N\in\mathbb N$, $N \ge C m^4$.
    \begin{itemize}
        \item There exists a set $\Lambda\subset\mathbb Z$ with the following
            property. For any $k\in\{-N,\ldots,N\}$ there exists $h\in\mathbb Z$,
    $h\ne0$, such that $k\pm lh\in\Lambda$ for $l=1,\ldots,m$.
\item There exists a set $\Lambda\subset\mathbb Z_N$ with the following
    property. For any $k\in\mathbb Z_N$ there exists $h\in\mathbb Z_N^*$, such
            that $k\pm lh\in\Lambda$ for $l=1,\ldots,m$.
    \end{itemize}
    Moreover, in both cases we have
    $$
    |\Lambda|\le C N^{1-\frac{1}{2m}}\log^{\frac1m}N.
    $$
\end{lemma}

Let us prove Theorem~\ref{thm_trig}. We will consider only the discrete case;
the continuous case is completely analogous.
\begin{proof}
    Pick some number $m$, apply Lemma~\ref{lem_set} to get a set $\Lambda$
    and Lemma~\ref{lem_poly} to get a polynomial $T\in\mathcal T_m^N$. In order to
    approximate $e(kx/N)$ by the space $\mathcal T(\Lambda)$ of trigonometric
    polynomials with spectrum in $\Lambda$, we find $h\in\mathbb Z_N^*$ such that $k\pm
    h,\ldots,k\pm mh\in\Lambda$. Consider the polynomial
$$
    t(x) = \sum_{l=-m}^m \hat T_m(l)e((k+lh)x/N).
$$
It equals $e(kx/N)+s(x)$, where $\supp\widehat{s}\subset\Lambda$, so we have to
    estimate $\|t\|_{L_p^N}$. Note that $t(x)=e(kx/N)T_m(hx)$,
$$
\|t\|_{L_p^N}= \|T_m(hx)\|_{L_p^N} = \|T_m(x)\|_{L_p^N} \ll m^{-\delta}.
$$
    It remains to take $m\asymp \log^{1/2} N$ and obtain
    $|\Lambda|\le N\exp(-(\log N)^{1/2+o(1)})$.
\end{proof}

Now we will prove Lemma~\ref{lem_set}.
\begin{proof}
We start with the discrete case.
    Let $\Lambda\subset\mathbb Z_N$ be random, each point gets in $\Lambda$ independent of others 
    with some probability $\tau\in(0,1)$.

Let us call a step $h\in\mathbb Z_N^*$ admissible for
a $k\in \mathbb Z_N$ if $k\pm lh\in\Lambda$ for $l=1,\ldots,m$. We have
$$
\P(\mbox{$h$ is admissible for $k$})=\tau^{2m}.
$$
Let us call $k$ ``bad'' if there are no admissible steps for $k$.
Assume that for any $k$ there are at least $S$ steps $h_1,\ldots,h_S\in\mathbb
    Z_N^*$ such that
the progressions $k\pm lh_i$, $l=1,\ldots,m$, do not intersect each other
(soon we will estimate $S$). Then the events that $h_i$ is admissible for $k$ are
independent for different $i$. Hence the probability that $k$ is bad is at
most the probability that none of $h_i$ are admissible, i.e.
$(1-\tau^{2m})^S$.

    The expected number of bad $k$ is at most $N(1-\tau^{2m})^S < N\exp(-S\tau^{2m})$ so with
    probability $\ge2/3$ this number does not exceed $3N\exp(-S\tau^{2m})$.
Moreover, $\E|\Lambda|=N\tau$, so with probability $\ge2/3$ we have
$|\Lambda|\le 3N\tau$. Hence there exists $\Lambda$ such that both estimates
are true. If $\tau$ is small enough to provide that
    \begin{equation}
        \label{S_bound}
        3N\exp(-S\tau^{2m}) < 1
    \end{equation}
then there will be no bad $k$ and the $\Lambda$ is the set we are looking
for.

It remains to estimate $S$ from below. There are totally $|\mathbb Z_N^*|\gg
N/\log N$ possible steps. A step $h_1$ forbids other steps $h$ that are
    solutions of $2m^2$ equations $lh \equiv \pm l_1h_1\pmod{N}$; $l,l_1\in\{1,\ldots,m\}$.
    Each equation has at most $\mathrm{gcd}(N,l)\le m$ solutions, so we can take
$$
S \asymp m^{-3}N/\log N.
$$
Finally, to satisfty~\eqref{S_bound} we take
$$
\tau = (\ln(4N) / S)^{\frac{1}{2m}} \ll (\log^2N/N)^{\frac{1}{2m}}.
$$

The continuous case is analogous. Now $\Lambda$ is a random set in
$\{-2N,\ldots,2N\}$, we consider $k\in\{-N,\ldots,N\}$. To ensure that
    the progressions $\{k\pm lh,\;l=1,\ldots,m\}$ do not intersect we take prime
numbers $h$ in the interval $(m,N/m]$; hence $S \asymp m^{-1}N/\log N$.
\end{proof}

\section{Further work}

In the most interesting case of the $L_1$ norm there is much more to do.

\begin{question}
    Is the trigonometric system $\{e(kx)\}_{1\le k\le N}$ rigid in $L_1$?
\end{question}

It is obvious that $d_n(\{\varphi_1,\ldots,\varphi_N\},L_1)\ge
A^{-1} d_n(B_1^N,\ell_\infty^N)$ for a uniformely bounded ($\|\varphi_k\|_\infty\le
A$) orthonormal system. So good approximation is impossible for $n\le c\log N$.
Can one improve this bound?
\begin{question}
    Prove that $d_n(\{w_1,\ldots,w_N\},L_1)\ge c$ for $n=n(N)$ such that
    $n/\log N\to\infty$.
\end{question}

The independence (or unconditionality) is sufficient for $L_1$-rigidity but this condition is, of
course, very restrictive. Can one get a weaker condition or at least prove the
rigidity of some non-independent systems of interest?

\begin{question}
    Is the Rademacher chaos $\{r_i(t)r_j(t)\}_{1\le i,j\le N}$ rigid in $L_1$?
\end{question}

This question is connected to the next one (see also
Statement~\ref{stm_explicit_rigid}):

\begin{question}
    Give explicit constructions of $\ell_1$-rigid sets in $W\subset \{-1,1\}^N$ of
    polynomial (or at least subexponential) size.
\end{question}

In the paper~\cite{DL19} the non-rigidity of DFT matrices was connected to the
non-rigidity of circulants. In terms of widths we get the following question.
\begin{question}
    Let $u\in[-1,1]^N$ and $T$ is the cyclic coordinate shift operator. Is it
    true that
    $$
    d_{N/2}(\{u,Tu,\ldots,T^{N-1}u\},\ell_1^N)=o(N)\;?
    $$
\end{question}

Let us proceed to the $L_0$ case.

\begin{question}
    Construct a uniformely bounded orthonormal system
    $\varphi_1,\ldots,\varphi_N\in D_N(0,1)$ that is $L_0$-rigid.
\end{question}

\begin{question}
    Is it true that for any finite abelian group $G$ the set of characters admits a
    good approximation by very-low-dimensional spaces:
    $$
    d_n(\widehat{G},L_0(G)) = o(1),\quad n = \exp(\log^{1-c}|G|)\;?
    $$
\end{question}

\end{document}